\documentclass{amsart}
\usepackage[all]{xy}
\usepackage{color}
\usepackage{amsthm}
\usepackage{amssymb}
\usepackage[colorlinks=true]{hyperref}
\usepackage{tikz}
\usetikzlibrary{matrix,arrows}

\numberwithin{equation}{section}

\newtheorem{theorem}[equation]{Theorem} 
\newtheorem*{theorem*}{Theorem}
\newtheorem{lemma}[equation]{Lemma}
\newtheorem{proposition}[equation]{Proposition}
\newtheorem{corollary}[equation]{Corollary}
\newtheorem*{corollary*}{Corollary}

\theoremstyle{remark}
\newtheorem{definition}[equation]{Definition}

\newtheorem{example}[equation]{Example}

\theoremstyle{remark}
\newtheorem{remark}[equation]{Remark}

\newcommand{\cA}{{\mathcal A}}
\newcommand{\cB}{{\mathcal B}}
\newcommand{\cC}{{\mathcal C}}
\newcommand{\cD}{{\mathcal D}}

\newcommand{\cO}{{\mathcal O}}

\newcommand{\Spt}{\mathrm{Spt}}

\newcommand{\add}{\mathsf{add}}

\newcommand{\bbL}{\mathbb{L}}

\newcommand{\bbS}{\mathbb{S}}

\newcommand{\bbQ}{\mathbb{Q}}
\newcommand{\bbZ}{\mathbb{Z}}

\DeclareMathOperator{\id}{id}

\DeclareMathOperator{\Fun}{Fun} 

\newcommand{\bbK}{I\mspace{-6.mu}K}

\newcommand{\dgcat}{\mathsf{dgcat}}

\newcommand{\perf}{\mathrm{perf}}

\newcommand{\dg}{\mathsf{dg}}

\newcommand{\Hom}{\mathrm{Hom}}

\newcommand{\rep}{\mathrm{rep}}

\newcommand{\Ho}{\mathrm{Ho}}

\newcommand{\Hmo}{\mathsf{Hmo}}
\newcommand{\op}{\mathsf{op}}

\newcommand{\too}{\longrightarrow}

\newcommand{\ie}{\textsl{i.e.}\ }
\newcommand{\eg}{\textsl{e.g.}}

\begin{document}

\title[Galois descent of additive invariants]{Galois descent of additive invariants}

\author{Gon{\c c}alo~Tabuada}

\address{Gon{\c c}alo Tabuada, Department of Mathematics, MIT, Cambridge, MA 02139, USA}
\email{tabuada@math.mit.edu}
\urladdr{http://math.mit.edu/~tabuada/}
\thanks{The author was partially supported by the NEC Award-2742738}
\date{\today}


\abstract{Making use of the recent theory of noncommutative motives, we prove that every additive invariant satisfies Galois descent. Examples include mixed complexes, Hochschild homology, cyclic homology, periodic cyclic homology, negative cyclic homology, connective algebraic $K$-theory, mod-$l$ algebraic $K$-theory, nonconnective algebraic $K$-theory, homotopy algebraic $K$-theory, topological Hochscild homology, and topological cyclic homology.}}

\maketitle 
\vskip-\baselineskip
\vskip-\baselineskip
\vskip-\baselineskip
\section{Introduction}\label{sec:introduction}
\subsection*{Additive invariants}
A {\em differential graded (=dg) category} $\cA$, over a base commutative ring $k$, is a category enriched over complexes of $k$-modules; see \S\ref{sec:dg}. Let us denote by $\dgcat$ the category of (small) dg categories. Every (dg) $k$-algebra $A$ gives naturally rise to a dg category $\underline{A}$ with a single object and (dg) $k$-algebra of endomorphisms $A$. Another source of examples is provided by $k$-schemes since the derived category of perfect complexes $\cD_\perf(X)$ of every quasi-compact separated $k$-scheme $X$ admits a canonical dg enhancement\footnote{In the particular case where $k$ is a field and $X$ is quasi-projective this dg enhancement is unique; see Lunts-Orlov \cite[Thm.~2.12]{LO}.} $\cD_\perf^\dg(X)$; see Keller \cite[\S4.6]{ICM-Keller}.

Given a dg category $\cA$, let us denote by $T(\cA)$ the dg category of pairs $(i,x)$, where $i\in \{1,2\}$ and $x$ is an object of $\cA$. The complex of morphisms in $T(\cA)$ from $(i,x)$ to $(i',x')$ is given by $\cA(x,x')$ if $i \geq i'$ and is zero otherwise. Composition is induced by the composition operation in $\cA$. Intuitively speaking, $T(\cA)$ ``dg categorifies'' the notion of upper triangular matrix.  Note that we have two inclusion dg functors $i_1: \cA \hookrightarrow T(\cA)$ and $i_2: \cA \hookrightarrow T(\cA)$.
\begin{definition}\label{def:additive}
Let $E:\dgcat\to \mathsf{D}$ be a functor with values in an additive category. We say that $E$ is an {\em additive invariant} if it satisfies the following two conditions:
\begin{itemize}
\item[(i)] it sends {\em Morita morphisms} (see \S\ref{sec:dg}) to isomorphisms;
\item[(ii)] given any dg category $\cA$, the inclusion dg functors induce an isomorphism\footnote{Condition (ii) can be equivalently formulated in terms of semi-orthogonal decompositions in the sense of Bondal-Orlov; see \cite[Thm.~6.3(4)]{IMRN}.}
$$ [E(i_1)\,\,E(i_2)]:E(\cA) \oplus E(\cA) \stackrel{\sim}{\too} E(T(\cA))\,.$$
\end{itemize}
\end{definition}
Examples of additive invariants include:
\begin{itemize}
\item[(i)] The mixed complex functor $C:\dgcat \to \cD(\Lambda)$ with values in the derived category of the dg $k$-algebra $\Lambda:=k[B]/B^2$ where $B$ is of degree $-1$ and $d(B)=0$; see Keller \cite[\S1.5 Thm. b) and \S1.12]{Exact}.
\item[(ii)] The Hochschild homology functor $HH: \dgcat \to \cD(k)$ (with values in the derived category of $k$), the cyclic homology functor $HC:\dgcat \to \cD(k)$, the periodic cyclic homology functor $HP:\dgcat \to \cD_{\bbZ/2}(k)$ (with values in the category of $\bbZ/2$-graded cochain complexes of $k$-modules), and the negative cyclic homology functor $HN:\dgcat \to \cD(k)$; see Keller \cite[\S2.2]{Exact2}.
\item[(iii)] The connective algebraic $K$-theory functor $K:\dgcat \to \Ho(\Spt)$ with values in the homotopy category of spectra; see Thomason-Trobaugh \cite[Thm.~1.9.8]{TT} and Waldhausen \cite[Thm.~1.4.2]{Wald}.
\item[(iv)] The mod-$l$ algebraic $K$-theory functor $K(-;\bbZ/l): \dgcat \to \Ho(\Spt)$, with $l\geq 2$ an integer; see \S\ref{sec:K-ths}. 
\item[(v)] The nonconnective algebraic $K$-theory functor $\bbK:\dgcat \to \Ho(\Spt)$; see Schlichting \cite[\S7 Thm. 4 and \S12.3 Prop. 3]{Negative}.
\item[(vi)] The homotopy algebraic $K$-theory functor $KH:\dgcat \to \Ho(\Spt)$; see \S\ref{sec:K-ths}.
\item[(vii)] The topological Hochschild homology functor $THH:\dgcat \to \Ho(\Spt)$ and the topological cyclic homology functor $TC:\dgcat \to \Ho(\Spt)$; see Blumberg-Mandell \cite[Prop. 3.10 and Thm.~10.8]{BM} and \cite[Prop.~8.9]{MacLane}.
\end{itemize}
\begin{remark}
When applied to $\underline{A}$, respectively to $\cD_\perf^\dg(X)$, the above additive invariants (i)-(vii) reduce to the classical invariants of (dg) $k$-algebras, respectively of $k$-schemes; consult \cite[Thm.~5.2]{ICM-Keller} for (i)-(ii), \cite[Thm.~5.1]{ICM-Keller} for (iii) and (v), \cite[Example~2.13]{Products} for (iv), \cite[Prop.~2.3]{Fundamental} for (vi), and \cite[Thm.~1.3]{BM} for (vii).
\end{remark}
\subsection*{Galois descent}
Let $k \subset l$ be a $G$-Galois extension of commutative rings in the sense of Auslander-Goldman \cite{AG} and Chase-Harrison-Rosenberg \cite{chase}; see \cite[\S0 Def.~1.5]{LNM}. Recall that $G$ is a finite subgroup of the $k$-algebra automorphisms of $l$ and that $k=l^G$. Let us write $m$ for the cardinality of $G$. Every $G$-Galois field extension is a $G$-Galois extension of commutative rings. Here is another example: 
\begin{example}
Let $k \subset l$ be a $G$-Galois field extension of number fields, \ie finite field extensions of $\bbQ$. Then, the corresponding extension $\cO_k \subset \cO_l$ of the rings of integers is a $G$-Galois extension of commutative rings if and only if $k \subset l$ is unramified; consult \cite[Thm.~4.1]{LNM} for details.
\end{example}
Given a dg category $\cA$, let us denote by $\cA_l$ the dg category obtained from $\cA$ by tensoring each complex of morphisms with $l$. Throughout the article all dg categories (\eg\ $\cA_l$) will be defined over the base commutative ring $k$. Note that $\cA_l$ comes equipped with an action of $G$ and with a canonical dg functor $\iota_\cA:\cA \to  \cA_l$. Our main result is the following:
\begin{theorem}{(Galois descent)}\label{thm:main}\
Assume that the $k$-module\footnote{Recall from \cite[Thm.~1.6]{LNM} that $l$ is a finitely generated projective $k$-module.} $l$ is {\em stably-free}, \ie that there exist integers $r,s >0$ such that $l \oplus k^{\oplus r} \simeq k^{\oplus s}$. Let $n:=s-r>0$ and $E:\dgcat \to \mathsf{D}$ an additive invariant with values in an idempotent complete $\bbZ[1/nm]$-linear category. Under these assumptions, one has a canonical isomorphism
\begin{equation*}
E(\iota_\cA): E(\cA) \stackrel{\sim}{\too} E(\cA_l)^G
\end{equation*}
for every dg category $\cA$. This isomorphism is moreover natural in $\cA$ and in $E$.
\end{theorem}
Recall from \cite[page~68]{Weibel} that when $k$ is a field, a local ring, or a principal ideal domain (PID), every finitely generated projective $k$-module is (stably-)free. Moreover, in the particular case where $k \subset l$ is a $G$-Galois field extension, one has $r=0$, $s=m$, and consequently $n=m$. Intuitively speaking, Theorem~\ref{thm:main} shows us that every additive invariant satisfies Galois descent as long as one inverts the rank of the $k$-module $l$ as well as the cardinality of $G$. The proof is based on the recent theory of noncommutative motives; consult \S\ref{sec:NCmotives}-\ref{sec:proof} for details. 
\begin{corollary}\label{cor:main1}
Let $A$ be a (dg) $k$-algebra and $X$ a quasi-compact separated $k$-scheme. When $\mathrm{char}(k) \nmid nm$, one has canonical natural isomorphisms
$$
\begin{array}{rcl}
C(A) \simeq C(A_l)^G && C(X) \simeq C(X_l)^G  \label{eq:descent00}\\
HH(A) \simeq HH(A_l)^G && HH(X) \simeq HH(X_l)^G  \label{eq:descent11}\\
HC(A) \simeq HC(A_l)^G && HC(X) \simeq HC(X_l)^G  \label{eq:descen12}\\
HP(A) \simeq HP(A_l)^G && HP(X) \simeq HP(X_l)^G   \label{eq:descent13}\\
HN(A) \simeq HN(A_l)^G && HN(X) \simeq HN(X_l)^G  \label{eq:descent14}\,,
\end{array}
$$
where $A_l:=A\otimes_k l$ and $X_l:=X \times_{\mathrm{Spec}(k)} \mathrm{Spec}(l)$. By passing to the associated homology groups one obtains 
$$
\begin{array}{rcl}
HH_\ast(A) \simeq HH_\ast(A_l)^G && HH_\ast(X) \simeq HH_\ast(X_l)^G  \nonumber\\
HC_\ast(A) \simeq HC_\ast(A_l)^G && HC_\ast(X) \simeq HC_\ast(X_l)^G   \nonumber\\
HP_\ast(A) \simeq HP_\ast(A_l)^G && HP_\ast(X) \simeq HP_\ast(X_l)^G    \nonumber\\
HN_\ast(A) \simeq HN_\ast(A_l)^G && HN_\ast(X) \simeq HN_\ast(X_l)^G   \nonumber\,.
\end{array}
$$
\end{corollary}
Let $(\Ho(\Spt)_{1/nm})^\natural$ be the idempotent completion of the localization $\Ho(\Spt)_{1/nm}$ of $\Ho(\Spt)$ at the $\bbZ[1/nm]$-linear stable equivalences and 
\begin{equation}\label{eq:loc}
(-)_{1/nm}: \Ho(\Spt) \too \Ho(\Spt)_{1/nm} \subset (\Ho(\Spt)_{1/nm})^\natural
\end{equation}
the associated localization functor. For every spectrum $S$ one has a canonical isomorphism $\pi_\ast(S_{1/nm}) \simeq \pi_\ast(S)_{1/nm}:= \pi_\ast(S) \otimes_\bbZ \bbZ[1/nm]$. Furthermore, by composing the above functors (iii)-(vii) with \eqref{eq:loc} one obtains additive invariants with values in an idempotent complete $\bbZ[1/nm]$-linear additive category.
\begin{corollary}\label{cor:main2}
For every (dg) $k$-algebra $A$ and for every quasi-compact separated $k$-scheme $X$ one has canonical natural isomorphisms
$$
\begin{array}{rcl}
K(A)_{1/nm} \simeq (K(A_l)_{1/nm})^G && K(X)_{1/nm} \simeq (K(X_l)_{1/nm})^G  \\
K(A;\bbZ/l)_{1/nm} \simeq (K(A_l;\bbZ/l)_{1/nm})^G && K(X;\bbZ/l)_{1/nm} \simeq (K(X_l;\bbZ/l)_{1/nm})^G  \\
\bbK(A)_{1/nm} \simeq (\bbK(A_l)_{1/nm})^G &&  \bbK(X)_{1/nm} \simeq (\bbK(X_l)_{1/nm})^G   \\
KH(A)_{1/nm} \simeq (KH(A_l)_{1/nm})^G &&  KH(X)_{1/nm} \simeq (KH(X_l)_{1/nm})^G   \\
THH(A)_{1/nm} \simeq (THH(A_l)_{1/nm})^G   && THH(X)_{1/nm} \simeq (THH(X_l)_{1/nm})^G  \label{eq:descent3}\\
TC(A)_{1/nm} \simeq (TC(A_l)_{1/nm})^G  && TC(X)_{1/nm} \simeq (TC(X_l)_{1/nm})^G  \label{eq:descent4}\,.
\end{array}
$$
By passing to the associated homotopy groups one obtains
$$
\begin{array}{rcl}
K_\ast(A)_{1/nm} \simeq (K_\ast(A_l)_{1/nm})^G && K_\ast(X)_{1/nm} \simeq (K_\ast(X_l)_{1/nm})^G  \\
K_\ast(A;\bbZ/l)_{1/nm} \simeq (K_\ast(A_l;\bbZ/l)_{1/nm})^G && K_\ast(X;\bbZ/l)_{1/nm} \simeq (K_\ast(X_l;\bbZ/l)_{1/nm})^G  \\
\bbK_\ast(A)_{1/nm} \simeq (\bbK_\ast(A_l)_{1/nm})^G &&  \bbK_\ast(X)_{1/nm} \simeq (\bbK_\ast(X_l)_{1/nm})^G   \\
KH_\ast(A)_{1/nm} \simeq (KH_\ast(A_l)_{1/nm})^G &&  KH_\ast(X)_{1/nm} \simeq (KH_\ast(X_l)_{1/nm})^G   \\
THH_\ast(A)_{1/nm} \simeq (THH_\ast(A_l)_{1/nm})^G   && THH_\ast(X)_{1/nm} \simeq (THH_\ast(X_l)_{1/nm})^G  \\
TC_\ast(A)_{1/nm} \simeq (TC_\ast(A_l)_{1/nm})^G  && TC_\ast(X)_{1/nm} \simeq (TC_\ast(X_l)_{1/nm})^G  \,.
\end{array}
$$
\end{corollary}
The isomorphism $HH_\ast(A) \simeq HH_\ast(A_l)^G$ was established by Geller-Weibel \cite{GW} in the particular case of a {\em commutative} $k$-algebra $A$. Under the assumption that $\mathrm{char}(k)=0$, the isomorphism $HC_\ast(A) \simeq HC_\ast(A_l)^G$ was also proved in \cite{GW}. In what concerns algebraic $K$-theory and $G$-Galois field extensions, the isomorphism $K_\ast(X)_{1/n}\simeq (K_\ast(X_l)_{1/n})^G$ goes back to Thomason's foundational work~\cite{Thomason}. A homotopy version of Galois descent for topological cyclic homology of {\em commutative} $k$-algebras was established by Tsalidis~\cite{T1}. Besides these examples in the ``commutative world'', all the remaining isomorphisms of Corollaries~\ref{cor:main1} and \ref{cor:main2} (and of Theorem~\ref{thm:main}) are, to the best of the author's knowledge, new in the literature. 

\medbreak\noindent\textbf{Acknowledgments:} The author is very grateful to Henri Gillet and John Rognes for their comments on a previous version of this article, and to the anonymous referee for all his comments, suggestions and corrections that greatly allowed the improvement of the article.
\section{Background on dg categories}\label{sec:dg}
Let $\cC(k)$ be the category of cochain complexes of $k$-modules; we use cohomological notation. A {\em differential graded (=dg) category $\cA$} is a category enriched over $\cC(k)$ (morphisms sets $\cA(x,y)$ are complexes) in such a way that composition fulfills the Leibniz rule $d(f \circ g) =d(f) \circ g +(-1)^{\mathrm{deg}(f)}f \circ d(g)$. A {\em dg functor} $F:\cA\to \cB$ is  a functor enriched over $\cC(k)$; consult Keller's ICM survey \cite{ICM-Keller} for further details.
\subsection*{Modules}
Let $\cA$ be a dg category. Its {\em opposite} dg category $\cA^\op$ has the same objects and complexes of morphisms given by $\cA^\op(x,y):=\cA(y,x)$. A {\em left} (resp. {\em right}) {\em $\cA$-module} $M$ is a dg functor $M:\cA\to \cC(k)_\dg$ (resp. $M:\cA^\op \to \cC_\dg(k)$) with values in the dg category $\cC_\dg(k)$ of cochain complexes of $k$-modules. Let us denote by $\cC(\cA)$ the category of right $\cA$-modules; see \cite[\S2.3]{ICM-Keller}. Recall from \cite[\S3.2]{ICM-Keller} that the {\em derived category $\cD(\cA)$ of $\cA$} is the localization of $\cC(\cA)$ with respect to the class of objectwise quasi-isomorphisms. Its full subcategory of compact objects (see \cite[Def.~4.2.7]{Neeman}) will be denoted by $\cD_c(\cA)$.
\subsection*{Morita morphisms}
A dg functor $F:\cA\to \cB$ is called a {\em Morita morphism} if the restriction of scalars functor $\cD(\cB) \stackrel{\sim}{\to} \cD(\cA)$ is an equivalence of (triangulated) categories; see \cite[\S4.6]{ICM-Keller}. As proved in \cite[Thm.~5.3]{IMRN}, $\dgcat$ admits a Quillen model structure whose weak equivalences are precisely the Morita morphisms. Let us denote by $\Hmo$ the homotopy category hence obtained.
\subsection*{Products}\label{sub:tensor}
The {\em cartesian product $\cA\times \cB$} (resp. {\em tensor product $\cA\otimes_k\cB$}) of two dg categories $\cA$ and $\cB$ is defined as follows: the set of objects is the cartesian product of the sets of objects of $\cA$ and $\cB$ and the complexes of morphisms are given by $(\cA\times \cB)((x,z),(y,w)):= \cA(x,y) \times \cB(z,w)$ (resp. by $(\cA\otimes_k\cB)((x,z),(y,w)):= \cA(x,y) \otimes_k \cB(z,w)$). As explained in \cite[\S2.3]{ICM-Keller}, the tensor product gives rise to a symmetric monoidal structure on $\dgcat$ with $\otimes$-unit the dg category $\underline{k}$. After deriving it $-\otimes^\bbL-$, this symmetric monoidal structure descends to $\Hmo$; consult \cite[\S4.2-4.3]{ICM-Keller} for details.
\subsection*{Bimodules}\label{sub:bimodules}
Let $\cA$ and $\cB$ be two dg categories. An {\em $\cA\text{-}\cB$-bimodule $\mathsf{B}$} is a dg functor $\mathsf{B}:\cA \otimes^\bbL_k \cB^\op\to \cC_\dg(k)$, \ie a right $(\cA^\op \otimes^\bbL_k \cB)$-module. A standard example is given by the $\cA\text{-}\cA$-bimodule
\begin{eqnarray}\label{eq:bimodule1}
\cA(-,-):\cA\otimes^\bbL_k\cA^\op \too \cC_\dg(k) && (x,y) \mapsto \cA(y,x)\,.
\end{eqnarray} 
Let us denote by $\rep(\cA,\cB)$ the full triangulated subcategory of $\cD(\cA^\op \otimes^\bbL_k \cB)$ consisting of those $\cA\text{-}\cB$-bimodules $\mathsf{B}$ such that for every object $x$ of $\cA$ the right $\cB$-module $\mathsf{B}(x,-)$ belongs to $\cD_c(\cB)$. Associated to a dg functor $F:\cA \to \cB$ we have the $\cA\text{-}\cB$-bimodule
\begin{eqnarray*}
{}_F\mathsf{B}_{\id}:\cA\otimes^\bbL_k \cB^\op \too \cC_\dg(k) && (x,z) \mapsto \cB(z,F(x))
\end{eqnarray*}
as well as the $\cB\text{-}\cA$-bimodule
\begin{eqnarray*}
{}_{\id}\mathsf{B}_F:\cB \otimes^\bbL_k \cA^\op \too \cC_\dg(k) && (z,x) \mapsto \cB(F(x),z)\,.
\end{eqnarray*}
Clearly, ${}_F\mathsf{B}_{\id}$ belongs to $\rep(\cA,\cB)$. In contrast, ${}_{\id}\mathsf{B}_F$ belongs to $\rep(\cB,\cA)$ if and only if for every $z \in \cB$ the right $\cA$-module $x \mapsto \cB(F(x),z)$ belongs to $\cD_c(\cA)$.
\begin{remark}\label{rk:key}
Given a dg category $\cD$ and a dg functor $F:\cA\to \cB$, we have an equality ${}_{\id}\mathsf{B}_{(\id \otimes^\bbL F)} = \cD(-,-)\otimes^\bbL_k {}_{\id} \mathsf{B}_F$ of $(\cD\otimes^\bbL_k \cB) \text{-}(\cD\otimes^\bbL_k \cA)$-bimodules, where $\id \otimes F$ stands for the dg functor obtained by tensoring $\cD$ with $F$.
\end{remark}
\section{Algebraic $K$-theories}\label{sec:K-ths}
\subsection*{Mod-$l$ algebraic $K$-theory}
Consider the following triangle in $\Ho(\Spt)$
$$ \bbS \stackrel{\cdot l}{\too} \bbS \too \bbS/l \too \Sigma (\bbS)\,,$$
where $l \ge 2$ is an integer, $\bbS$ is the sphere spectrum, and $\cdot l$ is the $l$-fold multiple of the identity morphism. Following Browder \cite{Browder} (and Karoubi), the {\em mod-$l$ algebraic $K$-theory functor} is defined as
\begin{eqnarray}\label{eq:mod}
K(-;\bbZ/l): \dgcat \too \Ho(\Spt) && \cA \mapsto K(\cA) \wedge^\bbL \bbS/l\,.
\end{eqnarray}
Since connective algebraic $K$-theory is an additive invariant and the functor $-\wedge^\bbL \bbS/l$ is additive, one observes that \eqref{eq:mod} is also an additive invariant.
\subsection*{Homotopy algebraic $K$-theory}
Consider the following simplicial $k$-algebra 
\begin{equation*}
\Delta_n:= k[t_0, \ldots, t_n]/ \big(\sum_i t_i -1\big) \qquad n \geq 0\,,
\end{equation*}
with faces and degenerancies given by the formulae
\begin{eqnarray*}
\partial_r(t_j) := \left\{ \begin{array}{lcr}
t_j & \text{if} & j <r \\
0 & \text{if} & j =r \\
t_{j-1} & \text{if} & j > r \\
\end{array} \right.
&
&
\delta_r(t_j) := \left\{ \begin{array}{lcr}
t_j & \text{if} & j <r \\
t_j + t_{j+1} & \text{if} & j =r \\
t_{j+1} & \text{if} & j > r \\
\end{array} \right.\,.
\end{eqnarray*}
The {\em homotopy algebraic $K$-theory functor} is defined as 
\begin{eqnarray}\label{eq:homotopy}
KH:\dgcat \too \Ho(\Spt) && \cA \mapsto \mathrm{hocolim}_n \,\bbK(\cA \otimes_k \underline{\Delta_n})\,.
\end{eqnarray}
\begin{proposition}
The above functor \eqref{eq:homotopy} is an additive invariant.
\end{proposition}
\begin{proof}
We start by verifying condition (i) of Definition \ref{def:additive}. Let $F:\cA \to \cB$ be a Morita morphism. Note that the $k$-algebras $\Delta_n, n \geq 0$, are $k$-flat. This implies that the dg functors $F\otimes \id : \cA \otimes_k \underline{\Delta_n} \to \cB \otimes_k \underline{\Delta_n}, n\geq 0$, are also Morita morphisms. Using the fact that the nonconnective algebraic $K$-theory functor satisfies condition (i) of Definition \ref{def:additive}, one hence obtains an induced isomorphism $KH(\cA) \stackrel{\sim}{\to} KH(\cB)$.

Let us now verify condition (ii). One needs to show that the inclusion dg functors $i_1, i_2: \cA \hookrightarrow T(\cA)$ induce an isomorphism
\begin{equation}\label{eq:induced}
\left(\mathrm{hocolim}_n\, \bbK(\cA\otimes_k \underline{\Delta_n})\right)^{\oplus 2}  \too \mathrm{hocolim}_n\, \bbK(T(\cA)\otimes_k \underline{\Delta_n})\,.
\end{equation} 
Under the canonical identification $T(\cA) \otimes_k \underline{\Delta_n} = T(\cA \otimes_k \underline{\Delta_n})$, the morphisms 
\begin{eqnarray*}
i_1 \otimes \id: \cA \otimes \underline{\Delta_n} \hookrightarrow T(\cA) \otimes_k \underline{\Delta_n} &\mathrm{and}& i_2 \otimes \id: \cA \otimes \underline{\Delta_n} \hookrightarrow T(\cA) \otimes_k \underline{\Delta_n}
\end{eqnarray*}
correspond to the two inclusions of $\cA\otimes_k \underline{\Delta_n}$ into $T(\cA\otimes_k \underline{\Delta_n})$. Moreover, since the left-hand-side of \eqref{eq:induced} identifies with $\mathrm{hocolim}_n\,\left(\bbK(\cA\otimes_k \underline{\Delta_n})^{\oplus 2}\right)$, the above induced morphism \eqref{eq:induced} reduces to 
\begin{eqnarray}\label{eq:morph-final}
\mathrm{hocolim}_n\, \left(\bbK(\cA\otimes_k \underline{\Delta_n})^{\oplus 2}\right) \too \mathrm{hocolim}_n\, \bbK(T(\cA\otimes_k \underline{\Delta_n}))\,.
\end{eqnarray}
Now, using the fact that the nonconnective algebraic $K$-theory functor satisfies condition (ii) of Definition \ref{def:additive}, one concludes that \eqref{eq:morph-final} is an isomorphism. This achieves the proof. 
\end{proof}
\section{Noncommutative motives}\label{sec:NCmotives}
In this section we recall from \cite{IMRN} the construction of the additive category $\Hmo_0$ of noncommutative motives; consult also the survey article \cite{survey}. This category, as well as its universal property, will play a key role in the proof of Theorem~\ref{thm:main}.

Let us consider first the homotopy category $\Hmo$. As proved in \cite[Cor.~5.10]{IMRN}, one has a canonical bijection $\Hom_{\Hmo}(\cA,\cB)\simeq \mathrm{Iso}\, \rep(\cA,\cB)$, where $\mathrm{Iso}$ denotes the set of isomorphism classes. Under this bijection, the composition law corresponds to the derived tensor product of bimodules
\begin{eqnarray*}
\mathrm{Iso}\, \rep(\cA,\cB) \times \mathrm{Iso}\, \rep(\cB,\cC) \too \mathrm{Iso}\, \rep(\cA,\cC) && (\mathsf{B},\mathsf{B}') \mapsto \mathsf{B} \otimes^{\bbL}_\cB \mathsf{B}'\,.
\end{eqnarray*}
Note that the identity of every $\cA \in \Hmo$ is given by the above $\cA\text{-}\cA$-bimodule \eqref{eq:bimodule1} and that we have a natural functor 
\begin{eqnarray}\label{eq:nat1}
\dgcat \too \Hmo && (F:\cA\to \cB) \mapsto {}_F\mathsf{B} _{\id} \,.
\end{eqnarray}
As explained in {\em loc. cit.}, the symmetric monoidal structure on $\dgcat$ descends to $\Hmo$ making \eqref{eq:nat1} into a symmetric monoidal functor. 

The {\em additivization of $\Hmo$} is the additive category $\Hmo_0$ with the same objects as $\Hmo$ and with abelian groups of morphisms given by $\Hom_{\Hmo_0}(\cA,\cB):=K_0\,\rep(\cA,\cB)$, where $K_0\, \rep(\cA,\cB)$ stands for the Grothendieck group of the triangulated category $\rep(\cA,\cB)$. The composition law is induced by the derived tensor product of bimodules
\begin{eqnarray*}
K_0\, \rep(\cA,\cB) \times K_0\, \rep(\cB,\cC) \too K_0\, \rep(\cA,\cC) && ([\mathsf{B}],[\mathsf{B}']) \mapsto [\mathsf{B} \otimes^{\bbL}_\cB \mathsf{B}']\,.
\end{eqnarray*}
Note that the products (=direct sums) in $\Hmo_0$ are given by the cartesian product of dg categories. Note also that we have a natural functor
\begin{eqnarray}\label{eq:nat2}
\Hmo \too \Hmo_0 && \mathsf{B} \mapsto [\mathsf{B}]\,.
\end{eqnarray}
As explained in {\em loc. cit.}, the symmetric monoidal structure on $\Hmo$ descends furthermore to a bilinear symmetric monoidal structure on $\Hmo_0$ making \eqref{eq:nat2} into a symmetric monoidal functor. As proved in \cite[Thms.~5.3 and 6.3]{IMRN}, the composition
$$ U:\dgcat \stackrel{\eqref{eq:nat1}}{\too} \Hmo \stackrel{\eqref{eq:nat2}}{\too} \Hmo_0$$
is the {\em universal additive invariant}, i.e. given any additive category $\mathsf{D}$ there is an induced equivalence of categories
\begin{equation}\label{eq:categories}
U^\ast: \Fun_{\add}(\Hmo_0,\mathsf{D}) \stackrel{\sim}{\too} \Fun_{\mathsf{A}}(\dgcat,\mathsf{D})\,,
\end{equation}
where the left-hand-side denotes the category of additive functors and the right-hand-side the category of additive invariants in the sense of Definition \ref{def:additive}. Because of this universal property, which is reminiscent from motives, $\Hmo_0$ is called the (additive\footnote{A triangulated version also exists in the literature; see \cite{Duke}.}) category of {\em noncommutative motives}. 
\section{Proof of Theorem~\ref{thm:main}}\label{sec:proof}
Let us denote by $\iota:k \to l$ the $G$-Galois extension of commutative rings and by $\underline{\iota}: \underline{k} \to \underline{l}$ the associated dg functor. 
\begin{lemma}\label{lem:key1}
The $\underline{l}\text{-}\underline{k}$-bimodule ${}_{\id} \mathsf{B}_{\underline{\iota}}$ belongs to $\rep(\underline{l},\underline{k})$. Consequently, it gives rise to a well-defined morphism $[{}_{\id} \mathsf{B}_{\underline{\iota}}]: U(\underline{l}) \to U(\underline{k})$ in $\Hmo_0$.
\end{lemma}
\begin{proof}
Note that the $\underline{l}\text{-}\underline{k}$-bimodule ${}_{\id} \mathsf{B}_{\underline{\iota}}$ is $l$ (considered as a complex of $k$-modules spaces concentrated in degree zero) endowed with the left multiplication by $l$ and with the right multiplication by $k$. Since by hypothesis $k \subset l$ is a $G$-Galois extension, $l$ is a finitely generated projective (right) $k$-module; see \cite[Thm.~1.6]{LNM}. This implies that $l$ is a compact object of the triangulated category $\cD(\underline{k})$ and consequently that ${}_{\id} \mathsf{B}_{\underline{\iota}}$ belongs to $\rep(\underline{l},\underline{k})$
\end{proof}
\begin{proposition}\label{prop:key2}
The morphisms $[{}_{\id}\mathsf{B}_{\underline{\iota}}]: U(\underline{l}) \to U(\underline{k})$ and $U(\underline{\iota}):U(\underline{k}) \to U(\underline{l})$ in $\Hmo_0$ verify the equality $[{}_{\id}\mathsf{B}_{\underline{\iota}}] \circ U(\underline{\iota})=n \cdot \id_{U(\underline{k})}$.
\end{proposition}
\begin{proof}
As explained in \S\ref{sec:NCmotives}, the composition
\begin{equation*}
U(\underline{k}) \stackrel{U(\underline{\iota})}{\too} U(\underline{l}) \stackrel{[{}_{\id}\mathsf{B}_{\underline{\iota}}]}{\too} U(\underline{k})
\end{equation*}
is given by $[{}_{\underline{\iota}} \mathsf{B}_{\id} \otimes^\bbL_{\underline{l}} {}_{\id} \mathsf{B}_{\underline{\iota}}] \in K_0\, \rep(\underline{k},\underline{k})$. Note that the $\underline{k}\text{-}\underline{k}$-bimodule ${}_\iota \mathsf{B}_{\id} \otimes^\bbL_{\underline{l}} {}_{\id} \mathsf{B}_{\underline{\iota}}$ is $l \otimes^\bbL_l l \simeq l \otimes_l l \simeq l$ (considered as a complex of $k$-modules concentrated in degree zero) endowed with the left and right multiplication by $k$. Recall from \S\ref{sec:NCmotives} that the identity $\id_{U(\underline{k})}$ of $U(\underline{k}) \in \Hmo_0$ is given by the class $[\underline{k}(-,-)]\in K_0\,\rep(\underline{k},\underline{k})$. 

Now, note that one has a canonical $\otimes$-equivalence of categories $\rep(\underline{k},\underline{k}) \simeq \cD_c(k)$. As a consequence, one obtains canonical ring isomorphisms
$$
K_0\, \rep(\underline{k},\underline{k}) \simeq K_0\cD(k) \simeq K_0(k)\,.
$$
Under these ring isomorphisms, $[{}_{\underline{\iota}} \mathsf{B}_{\id} \otimes^\bbL_{\underline{l}} {}_{\id} \mathsf{B}_{\underline{\iota}}]$ corresponds to $[l]$ and $[\underline{k}(-,-)]$ to $[k]$. Hence, it suffices to show the equality $[l] = n \cdot [k]$ in $K_0(k)$. Since by hypothesis $l$ is a stably-free $k$-module, there exist integers $r,s >0$ such that $l \oplus k^{\oplus r} \simeq k^{\oplus s}$. By construction of the Grothendieck group $K_0(k)$ we then obtain the equalities
$$ [l] = [k^{\oplus s}] - [k^{\oplus r}] = [k^{\oplus (s-r)}] = [k^{\oplus n}] = n \cdot [k]\,.$$
This achieves the proof.
\end{proof}
\begin{remark}
The proof of Proposition~\ref{prop:key2} shows that if $l$ is a stably-free $k$-module, then $[l]=n \cdot [k]$ for some integer $n >0$. As explained in \cite[page~68]{Weibel}, the converse is also true.
\end{remark}

%
%
%

Given an element $\sigma\in G$, let $\underline{\sigma}: \underline{l} \to \underline{l}$ be the associated dg functor.
\begin{proposition}\label{prop:key1}
The morphisms $[{}_{\id} \mathsf{B}_{\underline{\iota}}]: U(\underline{l}) \to U(\underline{k})$ and $U(\underline{\iota}):U(\underline{k}) \to U(\underline{l})$ in $\Hmo_0$ verify the equality $U(\underline{\iota}) \circ [{}_{\id} \mathsf{B}_{\underline{\iota}}] = \sum_{\sigma \in G} U(\underline{\sigma})$.
\end{proposition}
\begin{proof}
As explained in \S\ref{sec:NCmotives}, the composition
$$ U(\underline{l}) \stackrel{[{}_{\id} \mathsf{B}_{\underline{\iota}}]}{\too} U(\underline{k}) \stackrel{U(\underline{\iota})}{\too} U(\underline{l})$$\
is given by $[{}_{\id} \mathsf{B}_{\underline{\iota}} \otimes^\bbL_{\underline{k}} {}_{\underline{\iota}} \mathsf{B}_{\id}] \in K_0 \, \rep(\underline{l},\underline{l})$. Note that the $\underline{l}\text{-}\underline{l}$-bimodule ${}_{\id} \mathsf{B}_{\underline{\iota}} \otimes^\bbL_{\underline{k}} {}_{\underline{\iota}} \mathsf{B}_{\id}$ is $l \otimes^\bbL_k l \simeq l \otimes_k l$ (considered as a complex of $k$-modules concentrated in degree zero) endowed with the left and right multiplication by $l$.

Now, recall from \cite[\S0 Def.~1.5]{LNM} that since by hypothesis $k \subset l$ is a $G$-Galois extension of commutative rings, we have the following isomorphism of $l$-algebras
\begin{eqnarray*}
\theta: l \otimes_k l \stackrel{\sim}{\too} \prod_{\sigma \in G} l && b_1 \otimes b_2 \mapsto (b_1 b_2, \ldots, \sigma(b_1) b_2, \ldots )\,,
\end{eqnarray*}
where $l$ acts on the right factor of $l \otimes_k l$. By $\sigma$-twisting the left (diagonal) action of $G$ on $\prod_{\sigma \in G} l$, we hence obtain  the following isomorphism of $\underline{l}\text{-}\underline{l}$-bimodules
\begin{equation}\label{eq:iso-bimodules}
{}_{\id} \mathsf{B}_{\underline{\iota}} \otimes^\bbL_{\underline{k}} {}_{\underline{\iota}} \mathsf{B}_{\id} \simeq l \otimes_k l \stackrel{\theta}{\simeq} \prod_{\sigma \in G} {}_{\underline{\sigma}}\mathsf{B}_{\id}\,.
\end{equation}
By construction of the Grothendieck group we have $[\prod_{\sigma \in G} {}_{\underline{\sigma}}\mathsf{B}_{\id}] = \sum_{\sigma \in G} [{}_{\underline{\sigma}}\mathsf{B}_{\id}]$ in $K_0\, \rep(\underline{l},\underline{l})$. Therefore, using the above isomorphism \eqref{eq:iso-bimodules} and the fact that $[{}_{\underline{\sigma}}\mathsf{B}_{\id}] = U(\underline{\sigma})$, we obtain the desired equality $U(\underline{\iota}) \circ [{}_{\id} \mathsf{B}_{\underline{\iota}}] = [{}_{\id} \mathsf{B}_{\underline{\iota}} \otimes^\bbL_{\underline{k}} {}_{\underline{\iota}} \mathsf{B}_{\id}] = \sum_{\sigma \in G} U(\underline{\sigma})$.
\end{proof}
By definition, $\cA_l$ identifies with the dg category $\cA\otimes_k\underline{l}$ (where $l$ is considered as a $k$-algebra) and $\iota_\cA$ with the dg functor $\id\otimes \underline{\iota}: \cA\simeq \cA\otimes_k \underline{k} \to \cA\otimes_k \underline{l}$. Note that since $l$ is a projective (and hence $k$-flat) $k$-module there is no need to take the derived tensor product. Consider the $(\cA\otimes_k \underline{l}) \text{-}\cA$-bimodule $\cA(-,-)\otimes^\bbL_k {}_{\id} \mathsf{B}_{\underline{\iota}}$. Thanks to Remark~\ref{rk:key} (with $\cD:=\cA$ and $F:=\underline{\iota}$), we have $\cA(-,-)\otimes^\bbL_k {}_{\id} \mathsf{B}_{\underline{\iota}}={}_{\id} \mathsf{B}_{\iota_{\cA}}$. Moreover, since the $\cA\text{-}\cA$-bimodule $\cA(-,-)$ is the identity of $\cA\in\Hmo$ and the category $\Hmo$ is symmetric monoidal, we conclude from Lemma~\ref{lem:key1} that ${}_{\id} \mathsf{B}_{\iota_{\cA}}$ belongs to $\rep(\cA\otimes_k \underline{l}, \cA)$. By applying the symmetric monoidal functor $U(-)$ to $\id \otimes \underline{\iota}$ and the symmetric monoidal functor \eqref{eq:nat2} to ${}_{\id} \mathsf{B}_{\iota_\cA}$, we hence obtain from Propositions~\ref{prop:key2} and \ref{prop:key1}, and from the bilinearity of the symmetric monoidal structure on $\Hmo_0$, the following equalities
\begin{eqnarray*}
[{}_{\id}\mathsf{B}_{\iota_\cA}] \circ U(\iota_\cA) = n \cdot \id_{U(\cA)} &\mathrm{and} & U(\iota_\cA) \circ [{}_{\id} \mathsf{B}_{\iota_\cA}] = \sum_{\sigma \in G} U(\sigma_\cA)\,,
\end{eqnarray*}
where $\sigma_\cA$ stands for the dg functor $\id \otimes \underline{\sigma}: \cA\otimes_k \underline{l} \to \cA\otimes_k \underline{l}$.

We now have all the ingredients needed for the conclusion of the proof of Theorem~\ref{thm:main}. By hypothesis, $E:\dgcat \to \mathsf{D}$ is an additive invariant. Therefore, thanks to the above equivalence of categories \eqref{eq:categories}, there exists an additive functor $\overline{E}$ making the following diagram commutative
\begin{equation}\label{eq:diagram}
\xymatrix{
\dgcat \ar[d]_-U \ar[r]^-E & \mathsf{D} \\
\Hmo_0\ar[ur]_-{\overline{E}}& \,.
}
\end{equation}
We hence obtain well-defined morphisms
\begin{eqnarray*}
E(\iota_\cA): E(\cA)\too E(\cA\otimes_k \underline{l}) & & \overline{E}([{}_{\id} \mathsf{B}_{\iota_\cA}]): E(\cA\otimes_k \underline{l}) \too E(\cA)
\end{eqnarray*}
verifying the equalities
\begin{eqnarray}\label{eq:equalities}
\overline{E}([{}_{\id} \mathsf{B}_{\iota_\cA}])\circ E(\iota_\cA)=n \cdot \id_{E(\cA)} &\mathrm{and}& E(\iota_\cA) \circ \overline{E}([{}_{\id} \mathsf{B}_{\iota_\cA}])= \sum_{\sigma \in G}E(\sigma_\cA)\,.
\end{eqnarray}
Since by hypothesis the additive category $\mathsf{D}$ is idempotent complete and $\bbZ[1/m]$-linear, the $G$-invariant part $E(\cA\otimes_k \underline{l})^G$ of $E(\cA\otimes_k \underline{l})$ exists and is given by the image of the idempotent endomorphism $\frac{1}{m} \sum_{\sigma \in G} E(\sigma_\cA)$ of $E(\cA\otimes_k \underline{l})$. We obtain then a well-defined inclusion morphism $\mathrm{inc}:E(\cA\otimes_k \underline{l})^G \hookrightarrow E(\cA\otimes_k \underline{l})$ such that $E(\sigma_\cA) \circ \mathrm{inc}=\mathrm{inc}$ for every $\sigma \in G$. Note that the following equalities
$$ \left(\frac{1}{m} \sum_{\sigma \in G} E(\sigma_\cA)\right)\circ E(\iota_\cA)= \frac{1}{m} \sum_{\sigma\in G}E(\sigma_\cA \circ \iota_\cA) = \frac{1}{m} \sum_{\sigma\in G}E(\iota_\cA)= E(\iota_\cA)$$
imply that $E(\iota_\cA)$ factors through $\mathrm{inc}$. Let us now show that the composition and pre-composition of $E(\iota_\cA)$ with 
\begin{equation*}
 \frac{1}{m}\left(E(\cA\otimes_k \underline{l})^G \stackrel{\mathrm{inc}}{\hookrightarrow} E(\cA\otimes_k \underline{l}) \stackrel{\overline{E}([{}_{\id} \mathsf{B}_{\iota_\cA}])}{\too} E(\cA)\right)
\end{equation*}
is an isomorphism. Using the left-hand-side of \eqref{eq:equalities} and the $\bbZ[1/m]$-linearity of $\mathsf{D}$ we obtain the following equalities
\begin{equation*}
\frac{1}{m}\left(\overline{E}([{}_{\id} \mathsf{B}_{\iota_\cA}]\right) \circ \mathrm{inc}) \circ E(\iota_\cA) = \frac{1}{m}\cdot (n \cdot \id_{E(\cA)})=  \frac{n}{m} \cdot \id_{E(\cA)} \,.
\end{equation*}
On the other hand, by combining the equality $E(\sigma_\cA)\circ \mathrm{inc}=\mathrm{inc}$ with the right-hand-side of \eqref{eq:equalities}, we obtain
\begin{eqnarray*}
E(\iota_\cA) \circ \frac{1}{m}\left(\overline{E}([{}_{\id} \mathsf{B}_{\iota_\cA}]) \circ \mathrm{inc}\right) = \frac{1}{m}\cdot \left(m \cdot \id_{E(\cA\otimes_k \underline{l})}\right) = \id_{E(\cA\otimes_k \underline{l})^{G}} \,.
\end{eqnarray*}
The above two morphisms are isomorphisms. The first one follows from the fact that by hypothesis the category $\mathsf{D}$ is also $\bbZ[1/n]$-linear. The second one is the identity morphism. This allow us to conclude that $E(\iota_\cA)$ is also an isomorphism.

Finally, it remains only to show that the canonical isomorphism $E(\iota_\cA)$ is natural in $\cA$ and in $E$. This follows automatically from the following commutative diagrams
$$
\xymatrix{
\cA \ar[d]_-F \ar[r]^-\simeq & \cA\otimes_k \underline{k} \ar[d]^-{F\otimes \id} \ar[r]^-{\id \otimes \underline{\iota}} & \cA\otimes_k \underline{l} \ar[d]^-{F \otimes \id} \ar[r]^-{\id \otimes \underline{\sigma}} & \cA \otimes_k \underline{l} \ar[d]^-{F \otimes \id}\\
\cB \ar[r]_-{\simeq} & \cB \otimes_k \underline{k} \ar[r]_-{\id \otimes \underline{\iota}} & \cB \otimes_k \underline{l} \ar[r]_-{\id \otimes \underline{\sigma}} & \cB\otimes_k \underline{l} \,
}
$$
$$
\xymatrix{
E(\cA) \ar[r]^-\simeq \ar[d]_-{\eta_\cA} & E(\cA\otimes_k \underline{k}) \ar[rr]^-{E(\id \otimes\underline{\iota})} \ar[d]^-{\eta_{\cA\otimes_k \underline{k}}} && E(\cA \otimes_k \underline{l}) \ar[d]^-{\eta_{\cA\otimes_k \underline{l}}} \ar[rr]^-{E(\id \otimes \underline{\sigma})} && E(\cA \otimes_k \underline{l}) \ar[d]^-{\eta_{\cA\otimes_k \underline{l}}} \\
E'(A) \ar[r]_-\simeq & E'(\cA\otimes_k \underline{k}) \ar[rr]_{E'(\id \otimes \underline{\iota})} && E'(\cA \otimes_k \underline{l}) \ar[rr]_-{E'(\id \otimes \underline{\sigma})}&&  E'(\cA\otimes_k \underline{l})\,,
}
$$
where $F$ is any dg functor, $\sigma$ any element of $G$, and $\eta:E\Rightarrow E'$ any natural transformation between additive invariants. This concludes the proof of Theorem~\ref{thm:main}.
\section{Remaining proofs}
\subsection*{Corollary \ref{cor:main1}}

The idempotent complete categories $\cD(\Lambda), \cD(k)$ and $\cD_{\bbZ/2}(k)$ are $k$-linear. Hence, when $\mathrm{char}(k)\nmid nm$, they are in particular $\bbZ[1/nm]$-linear.

In what concerns quasi-compact separated $k$-schemes $X$, the right-hand-side isomorphisms of Corollary~\ref{cor:main1} follow from the following Morita morphisms
\begin{eqnarray}
\cD_\perf^\dg(X) \otimes_k \underline{l} & \simeq & \cD_\perf^\dg(X) \otimes_k \cD_\perf^\dg(\mathrm{Spec}(l)) \nonumber\\
& \simeq & \cD_\perf^\dg(X \times_{\mathrm{Spec}(k)} \mathrm{Spec}(l)) \label{eq:2} \\
& = &\cD_\perf^\dg(X_l)\,. \nonumber
\end{eqnarray}
The above Morita morphism \eqref{eq:2} was established in \cite[Prop.~6.2]{Regularity} in the case where $k$ is a field. However, since $l$ is a projective (and hence $k$-flat) $k$-module, the same proof holds.

\subsection*{Corollary \ref{cor:main2}}
The same argument, concerning quasi-compact separated $k$-schemes, applies.


\begin{thebibliography}{00}



\bibitem{AG} M.~Auslander and O.~Goldman, {\em The Brauer group of a commutative ring}. Trans. Amer. Math. Soc. {\bf 81} (1960), 367--409.

\bibitem{BM} A.~Blumberg and M.~Mandell, {\em Localization theorems in topological
Hochschild homology and topological cyclic homology}. Geometry and Topology {\bf 16} (2012), 1053--1120.    





   
\bibitem{Browder} W. Browder, {\em Algebraic $K$-theory with coefficients $\bbZ/p$}. Lecture Notes in Math. {\bf 657} (1978) 40--84.   

\bibitem{chase} S.~U.~Chase, D.~K.~Harrison and A.~Rosenberg, {\em Galois theory and Galois cohomology of commutative rings}. Mem. Amer. Math. Soc. {\bf 52} (1965).     
   
\bibitem{GW} S.~Geller and C.~Weibel, {\em {\'E}tale descent for Hochschild and cyclic homology}. Comment. Math. Helv. {\bf 66} (1991), 368--388.   

\bibitem{LNM} C.~Greither, {\em Cyclic Galois extensions of commutative rings}. 
Lecture Notes in Math. {\bf 1534} (1992). Springer-Verlag, Berlin. 
    

    
\bibitem{ICM-Keller} B.~Keller, {\em On differential graded
    categories}. International Congress of Mathematicians (Madrid), Vol.~II,
  151--190, Eur.~Math.~Soc., Z{\"u}rich, 2006.

\bibitem{Exact} \bysame, {\em On the cyclic homology of exact categories}. JPAA {\bf 136} (1999), no.~1, 1--56.   

\bibitem{Exact2} \bysame, {\em On the cyclic homology of ringed spaces and schemes}. Doc. Math. {\bf 3} (1998), 231--259.

  
  
\bibitem{LO} V.~Lunts and D.~Orlov, {\em Uniqueness of enhancement for triangulated categories}. J. Amer. Math. Soc. {\bf 23} (2010), no.~3, 853--908. 



\bibitem{Neeman} A.~Neeman, {\em Triangulated categories}. Annals of Math. Studies {\bf 148}. Princeton University Press, Princeton, NJ, (2001).    
 
\bibitem{Negative} M.~Schlichting, {\em Negative $\mbox{K}$-theory of
    derived categories}. Math.~Z. {\bf 253} (2006), no.~1, 97--134.        

\bibitem{Products} G.~Tabuada, {\em Products, multiplicative Chern characters, and finite coefficients via noncommutative motives}. JPAA {\bf 217} (2013), 1279--1293.

\bibitem{survey} \bysame, {\em A guided tour through the garden of noncommutative motives}. Clay Mathematics Proceedings {\bf 16} (2012), 259--276.

\bibitem{Fundamental} \bysame, {\em The fundamental theorem via derived Morita invariance, localization, and $\mathbf{A}^1$-homotopy invariance}. Journal of $K$-Theory {\bf 9} (2012), no.~3, 407--420.

\bibitem{MacLane} \bysame, {\em Generalized spectral categories, topological
Hochschild homology, and trace maps}, Algebraic and Geometric Topology {\bf 10} (2010), 137--213.

\bibitem{Duke} \bysame, {\em Higher $K$-theory via universal invariants}. Duke Math. J. {\bf 145} (2008), no.~1, 121--206.

\bibitem{IMRN} \bysame, {\em Additive invariants of dg categories}.  Int. Math. Res. Not. {\bf 53} (2005), 3309--3339.   

\bibitem{Regularity} \bysame, {\em $E_n$-regularity implies $E_{n-1}$-regularity}. Available at arXiv:1212.1112. To appear in Doc. Math.

\bibitem{Thomason} R.~W.~Thomason, {\em Algebraic K-theory and {\'e}tale cohomology}. Annales Scientifiques de l'{\'E}cole Normale Sup{\'e}rieure, $4^e$ s{\'e}rie, tome {\bf 18} (1985), no.~3, 437--552. 

\bibitem{TT} R.~W.~Thomason and T.~Trobaugh,
{\em Higher algebraic $K$-theory of schemes and of derived 
categories}. Grothendieck Festschrift, Volume III. Volume {\bf 88} of
Progress in Math., 247--436. Birkhauser, Boston, Bassel, Berlin, 1990. 

\bibitem{T1} S.~Tsalidis, {\em On the {\'e}tale descent problem for topological cyclic homology and algebraic K-theory}. $K$-Theory {\bf 21} (2000), no.~2, 151--199. 


\bibitem{Wald} F.~Waldhausen, {\em Algebraic K-theory of spaces}.  Algebraic and Geometric topology (New Brunswick, N.~J., 1983), 318--419, Lecture Notes in Math., {\bf 1126}, Springer, Berlin, 1985.


\bibitem{Weibel} C.~Weibel, {\em $K$-book}. Available at {\tt http://www.math.rutgers.edu/$\sim$weibel/Kbook.pdf}.

\end{thebibliography}
\end{document}